\documentclass[graybox,footinfo]{svmult}

\smartqed
\usepackage{mathptmx}       
\usepackage{helvet}         
\usepackage{courier}        
\usepackage{type1cm}        
\usepackage{graphicx}       

\usepackage{array,colortbl}
\usepackage{amsmath,amsfonts,amssymb,bm} 
\DeclareFontFamily{U}{mathx}{\hyphenchar\font45}
\DeclareFontShape{U}{mathx}{m}{n}{
	<5> <6> <7> <8> <9> <10>
	<10.95> <12> <14.4> <17.28> <20.74> <24.88>
	mathx10
}{}
\DeclareSymbolFont{mathx}{U}{mathx}{m}{n}
\DeclareFontSubstitution{U}{mathx}{m}{n}
\DeclareMathAccent{\widecheck}      {0}{mathx}{"71}

\renewcommand{\email}[1]{\emailname: #1} 

\usepackage{mathptmx}       
\usepackage{helvet}         
\usepackage{courier}        

\usepackage{makeidx}         
\usepackage{graphicx}        
\usepackage[bottom]{footmisc}

\usepackage{latexsym}
\usepackage{amsmath}
\usepackage{amsfonts}
\usepackage{amssymb}
\usepackage{bm}

\usepackage{url}
\usepackage{algorithm}
\usepackage{algorithmic}
\usepackage[misc,geometry]{ifsym}

\renewenvironment{proof}{\noindent{\itshape Proof.}}{\smartqed\qed}

\spdefaulttheorem{assumption}{Assumption}{\upshape \bfseries}{\itshape}
\spdefaulttheorem{algo}{Algorithm}{\upshape \bfseries}{\itshape}

\newcommand{\bsk}{{\boldsymbol{k}}}

\newcommand{\bsn}{{\boldsymbol{n}}}

\newcommand{\bsx}{{\boldsymbol{x}}}

\newcommand{\bslambda}{{\boldsymbol{\lambda}}}

\newcommand{\bbN}{{\mathbb{N}}}

\newcommand{\bbZ}{{\mathbb{Z}}}

\DeclareSymbolFont{bbold}{U}{bbold}{m}{n}
\DeclareSymbolFontAlphabet{\mathbbold}{bbold}

\newcommand{\bbone}{{\mathbbold{1}}}

\newcommand{\calA}{{\mathcal{A}}}
\newcommand{\calB}{{\mathcal{B}}}
\newcommand{\calC}{{\mathcal{C}}}

\newcommand{\calF}{{\mathcal{F}}}
\newcommand{\calG}{{\mathcal{G}}}
\newcommand{\calH}{{\mathcal{H}}}

\newcommand{\calO}{{\mathcal{O}}}

\providecommand{\argmax}{\operatorname*{argmax}}

\DeclareSymbolFont{bbold}{U}{bbold}{m}{n}
\DeclareSymbolFontAlphabet{\mathbbold}{bbold}

\usepackage{microtype} 

\usepackage[colorlinks=true,linkcolor=black,citecolor=black,urlcolor=black]{hyperref}
\usepackage{bookmark}
\makeatletter 
\providecommand*{\toclevel@author}{999}
\providecommand*{\toclevel@title}{0}
\makeatother

\usepackage{bbm,mathtools,array,longtable,booktabs,graphicx,color}

\newcommand{\DHJRnorm}[2][{}]{\ensuremath{\left \lVert #2 \right \rVert}_{#1}}
\newcommand{\DHJRbignorm}[2][{}]{\ensuremath{\bigl \lVert #2 \bigr \rVert}_{#1}}
\newcommand{\DHJRBignorm}[2][{}]{\ensuremath{\Bigl \lVert #2 \Bigr \rVert}_{#1}}
\newcommand{\DHJRabs}[1]{\ensuremath{{\left \lvert #1 \right \rvert}}}

\definecolor{orange}{rgb}{1.0,0.3,0.0}
\definecolor{violet}{rgb}{0.75,0,1}

\allowdisplaybreaks[4]

\begin{document}

\title*{An Adaptive Algorithm Employing Continuous Linear Functionals}
\author{Yuhan Ding \and Fred J. Hickernell \and Llu\'{\i}s Antoni Jim\'{e}nez Rugama}
\institute{Yuhan Ding \at MCA 310, Department of Mathematics, Misericordia University,\\ 301 Lake St., Dallas, PA, 18612 \email{yding@misericordia.edu}
\and
Fred J. Hickernell \at Center for Interdisciplinary Scientific Computation and \\
Department of Applied Mathematics, Illinois Institute of Technology \\ RE 208, 10 W.\ 32$^{\text{nd}}$ St., Chicago, IL 60616 \email{hickernell@iit.edu}
\and Llu\'{\i}s Antoni Jim\'{e}nez Rugama \at
Department of Applied Mathematics, Illinois Institute of Technology,\\ RE 208, 10 W.\ 32$^{\text{nd}}$ St., Chicago, IL 60616 \email{ljimene1@hawk.iit.edu}}

\maketitle

\abstract{Automatic algorithms attempt to provide approximate solutions that differ from exact solutions by no more than a user-specified error tolerance. This paper describes an automatic, adaptive algorithm for approximating the solution to a general linear problem on Hilbert spaces. The algorithm employs continuous linear functionals of the input function, specifically Fourier coefficients. We assume that the Fourier coefficients of the solution decay sufficiently fast, but do not require the decay rate to be known a priori. We also assume that the Fourier coefficients decay steadily, although not necessarily monotonically. Under these assumptions, our adaptive algorithm is shown to produce an approximate solution satisfying the desired error tolerance, without prior knowledge of the norm of the function to be approximated. Moreover, the computational cost of our algorithm is shown to be essentially no worse than that of the optimal algorithm. We provide a numerical experiment to illustrate our algorithm.}

\section{Introduction}
Adaptive algorithms determine the design and sample size needed to solve problems to the desired accuracy based on the input function data sampled.  A priori upper bounds on some norm of the input function are not needed, but some underlying assumptions about the input function are required for the adaptive algorithm to succeed.  Here we consider \emph{general linear problems} where a finite number of series coefficients of the input function are used to obtain an approximate solution.  The proposed algorithm produces an approximation with guaranteed accuracy.  Moreover, we demonstrate that the computational cost of our algorithm is essentially no worse than that of the best possible algorithm.  Our adaptive algorithm is defined on a \emph{cone} of input functions.

\subsection{Input and Output Spaces}  Let $\calF$ be a separable Hilbert space of inputs with orthonormal basis $\{u_i\}_{i \in \bbN}$, let $\calG$ be a separable Hilbert space of outputs  with orthonormal basis $\{v_i\}_{i\in \bbN}$, and let their norms be defined as the $\ell^2$-norms of their series coefficients:  
\begin{subequations}\label{probDef}
\begin{gather}
f = \sum_{i\in \bbN} \widehat{f}_i u_i \in \calF, \qquad 
\DHJRnorm[\calF]{f}=\DHJRbignorm[2]{(\widehat{f}_i\big)_{i\in \bbN}}, \\
g = \sum_{i\in \bbN} \widehat{g}_i v_i \in \calG, \qquad \DHJRnorm[\calG]{g}=\DHJRbignorm[2]{(\widehat{g}_i\big)_{i\in \bbN}}.
\end{gather}
Let these two bases be chosen so that the linear solution operator, $S:\calF \to \calG$, satisfies
\begin{gather}
S(u_i) = \lambda_i v_i, \quad i \in \bbN, \qquad S(f) = \sum_{i=1}^n \lambda_i \widehat{f}_i v_i, \\
\lambda_1 \ge \lambda_2 \ge \cdots > 0, \quad \lim_{i \to \infty} \lambda_i = 0, \qquad
\DHJRnorm[\calF \to \calG]{S} := \sup_{f \ne 0} \frac{\DHJRnorm[\calG]{S(f)}}{\DHJRnorm[\calF]{f}} = \lambda_1.
\end{gather}
\end{subequations}
This setting includes, for example, the recovery of functions, derivatives, indefinite integrals, and solutions of linear (partial) differential equations.  We focus on cases where the exact solution generally requires an infinite number of series coefficients, $\widehat{f}_i$.

The existence of the $( u_i )_{i \in \bbN}$, $ ( v_i )_{i \in \bbN}$, and $( \lambda_i )_{i \in \bbN}$ for a given $\calF$, $\calG$, and $S$ follows from the singular value decomposition.  The ease of identifying explicit expressions for these quantities depends on the particular problem of interest.  Alternatively, one may start with a choice of $( u_i )_{i \in \bbN}$, $ ( v_i )_{i \in \bbN}$, and $( \lambda_i )_{i \in \bbN}$, which then determine the solution operator, $S$, and the spaces $\calF$ and $\calG$.

\subsection{Solvability}  Let $\calH$ be any subset of $\calF$, and let $\calA(\calH)$ denote the set of deterministic algorithms that successfully approximate the solution operator $S : \calH \to \calG$ within some error tolerance for all inputs in $\calH$:
\begin{multline} \label{adapErrCrit}
\calA(\calH) : = \left\{ \text{algorithms } A:\calH \times (0,\infty) \rightarrow \calG : 
\right . \\ \left .
\DHJRbignorm[\calG]{S(f) - A(f,\varepsilon)} \le \varepsilon \ \forall f \in \calH, \ \varepsilon > 0 
\right \}.
\end{multline}
Algorithms in $\calA(\calH)$ are allowed to sample adaptively any bounded, linear functionals
of the input function.  They must sample only a finite number of linear functionals for each input function and positive tolerance.  The definition of $\calH$ can be used to construct algorithms in $\calA(\calH)$, but no other a priori knowledge about the input functions is available.  Following \cite{KunEtal19a} we call a problem \emph{solvable} for inputs $\calH$ if $\calA(\calH)$ is non-empty.

Our problem is not solvable for the whole Hilbert space $\calF$, as can be demonstrated by contradiction. For any potential algorithm, we show that there exists some $f \in \calF$, that looks like $0$ to the algorithm, but for which $S(f)$ is far from $S(0) = 0$.   Choose any $A \in \calA(\calF)$ and $\varepsilon > 0$, and let $L_1, \ldots, L_n$ be the linear functionals are used to compute $A(0,\varepsilon)$. Since the output space, $\calG$, is infinite dimensional and $n$ is finite, there exists some nonzero $f \in \calF$ satisfying that $L_1(f) = \cdots = L_n(f) = 0$ with non-zero $S(f)$. This means that $A(cf,\varepsilon) = A(0,\varepsilon)$ for any real $c$, and both of these have approximation error no greater than $\varepsilon$, i.e.,
\begin{align*}
    \varepsilon &\ge \frac 12 \left[ \DHJRnorm[\calG]{S(0) - A(0,\varepsilon)} + \DHJRnorm[\calG]{S(cf) - A(cf,\varepsilon)}\right] \\
    & =  \frac 12 \left[ \DHJRnorm[\calG]{0 - A(0,\varepsilon)} + \DHJRnorm[\calG]{S(cf) - A(0,\varepsilon)}\right] \\
    & \ge  \frac {\DHJRnorm[\calG]{cS(f)}}2 = \frac {\DHJRabs{c}\DHJRnorm[\calG]{S(f)}}2  \qquad \text{by the triangle inequality}.
\end{align*}
Since $S(f) \ne 0$, it is impossible for this inequality to hold for all real $c$.  The presumed $A$ does not exist, $\calA(\calF)$ is empty, and our problem is not solvable for $\calF$. However, it is solvable for well-chosen subsets of $\calF$, as will be shown in the sections below.

\subsection{Computational Cost of the Algorithm and Complexity of the Problem} The computational cost of an algorithm $A \in \calA(\calH)$ for $f \in \calH$ and error tolerance $\varepsilon$ is denoted $\textup{cost}(A,f,\varepsilon)$, and is defined as the number of linear functional values required to produce $A(f,\varepsilon)$.  By overloading the notation, we define the cost of algorithms for sets of inputs, $\calH$, as 
\begin{equation*}
\textup{cost}(A,\calH,\varepsilon) : = \sup \{\textup{cost}(A,f,\varepsilon) : f \in \calH \} \qquad \forall \varepsilon > 0.
\end{equation*}
For unbounded sets, $\calH$, this cost may be infinite.  Therefore, it is meaningful to define the cost of algorithms for input functions in $\calH \cap \calB_{\rho}$, where $\calB_{\rho} : = \{ f \in \calF : \DHJRnorm[\calF]{f} \le \rho \}$ is the ball of radius $\rho$:
\begin{equation*}
\textup{cost}(A,\calH,\varepsilon,\rho) : = \sup \{\textup{cost}(A,f,\varepsilon) : f \in \calH \cap \calB_{\rho} \} \qquad \forall \rho > 0, \ \varepsilon > 0.
\end{equation*}
Finally, we define the complexity of the problem as the computational cost of the best algorithm:
\begin{gather*}
\textup{comp}(\calA(\calH),\varepsilon) := \min_{A \in \calA(\calH)} \textup{cost}(A, \calH, \varepsilon), \\
\textup{comp}(\calA(\calH),\varepsilon,\rho) := \min_{A \in \calA(\calH)} \textup{cost}(A, \calH, \varepsilon, \rho).
\end{gather*}
Note that $\textup{comp}(\calA(\calH),\varepsilon,\rho) \ge \textup{comp}(\calA(\calH \cap \calB_{\rho}),\varepsilon)$.  In the former case, the algorithm is unaware  that the input function has norm no greater than $\rho$.  

An optimal algorithm for $\calB_{\rho}$ can be constructed in terms of interpolation with respect to the first $n$ series coefficients of the input, namely,
\begin{gather}  \label{optAdef}
A_n(f) := \sum_{i=1}^n \lambda_{i} \widehat{f}_{i} v_{i}, \\
\label{errOpt}
\DHJRnorm[\calG]{S(f) - A_n(f)} = \DHJRBignorm[2]{\left(\lambda_{i} \widehat{f}_{i} \right)_{i= n+1}^{\infty}} \le \lambda_{n+1} \DHJRnorm[\calF]{f}.
\end{gather}
Define the non-adaptive algorithm as
\begin{equation} \label{optBallalg}
\widehat{A}(f,\varepsilon) = A_{n^*}(f), \quad \text{where } n^* = \min\{ n : \lambda_{n+1} \le \varepsilon/\rho \}, \qquad \widehat{A} \in \calA(\calB_{\rho}).
\end{equation}
This algorithm is optimal among algorithms in $\calA(\calB_{\rho})$, i.e.,
\[
\textup{comp}(\calA(\calB_{\rho}),\varepsilon) = \textup{cost}(\widehat{A},\calB_{\rho},\varepsilon) =
\min\{ n : \lambda_{n+1} \le \varepsilon/\rho \}.
\]

To prove this, let $A^*$ be an arbitrary algorithm in $\calA(\calB_{\rho})$, and let $L_1, \ldots, L_N$ be the linear functionals chosen when evaluating this algorithm for the zero function with tolerance $\varepsilon$.  Thus, $A^*(0,\varepsilon)$ is some function of $(L_1(0) , \ldots, L_N(0)) = (0, \ldots, 0)$.  Let $f$ be a linear combination of $u_1, \ldots, u_{N+1}$ with norm $\rho$ satisfying  $L_1(f) = \cdots = L_N(f) = 0$, then $A^*(\pm f) = A^*(0)$, and
\begin{align*}
\varepsilon & \ge \max_{\pm} \DHJRnorm[\calG]{S(\pm f) - A^*(\pm f)} =  \max_{\pm} \DHJRnorm[\calG]{\pm S( f) - A^*(0)} \\
& \ge \frac 12 \left [ \DHJRnorm[\calG]{S( f) - A^*(0)} + \DHJRnorm[\calG]{- S( f) - A^*(0)}\right] \\
& \ge \DHJRnorm[\calG]{S(f)} 
= \DHJRBignorm[2]{\big(\lambda_i\widehat{f}_i\big)_{i=1}^{N+1}} \\
& \ge \lambda_{N+1} \DHJRBignorm[2]{\big(\widehat{f}_i\big)_{i=1}^{N+1}} = \lambda_{N+1} \DHJRnorm[\calF]{f} = \lambda_{N+1} \rho.
\end{align*}
Thus, $\lambda_{N+1} \le \varepsilon/\rho$, and 
\[
\textup{cost}(A^*,\calB_\rho,\varepsilon) \ge \textup{cost}(A^*,0,\varepsilon)  = N \ge \min\{ n : \lambda_{n+1} \le \varepsilon/\rho \} = \textup{cost}(\widehat{A},\calB_\rho,\varepsilon).
\]
Hence, algorithm $\widehat{A}$ defined in \eqref{optBallalg} is optimal for $\calA(\calB_{\rho})$.

\begin{example} Consider the case of function approximation for periodic functions defined over [0,1], and algorithm $\widehat{A}$ defined in \eqref{optBallalg}:
	\begin{align*}
	f &= \sum_{k \in \bbZ} \widehat{f}(k) \widehat{u}_{k}  = \sum_{i \in \bbN} \widehat{f}_i u_i, 
	& S(f) & = \sum_{k \in \bbZ} \widehat{f}(k) \widehat{\lambda}_k \widehat{v}_k = \sum_{i \in \bbN} \widehat{f}_i \lambda_i v_i, \\
	\widehat{v}_{k} (x) &:= \begin{cases} 1, & k = 0, \\
	\displaystyle \sqrt{2} \sin(2\pi k x), & k > 0, \\
	\displaystyle \sqrt{2} \cos(2\pi k x), & k < 0, \\
	\end{cases} 
	& v_i &= \begin{cases} \widehat{v}_{-i/2}, & i \text{ even},\\
	\widehat{v}_{(i-1)/2}, & i \text{ odd},
	\end{cases} \\
	\widehat{\lambda}_{k} &:= \begin{cases} 1, & k = 0, \\
	\displaystyle \frac{1}{\DHJRabs{k}^r}, & k \ne 0,
	\end{cases}
	& \lambda_i &= \widehat{\lambda}_{\lfloor i/2\rfloor} = \frac{1}{\max(1,\lfloor i/2\rfloor)^r}, \\
	\widehat{u}_{k} &:=  \widehat{\lambda}_k \widehat{v}_k,
	& u_i & = \lambda_i v_i = \begin{cases} \widehat{u}_{-i/2}, & i \text{ even},\\
	\widehat{u}_{(i-1)/2}, & i \text{ odd},
	\end{cases}
	\\
	&&\widehat{f}_i& = \begin{cases} \widehat{f}(-i/2), & i \text{ even},\\
	\widehat{f}((i-1)/2), & i \text{ odd},
	\end{cases}
	\end{align*}
	\begin{align*}
	\textup{comp}(\calA(\calB_{\rho}),\varepsilon) &= \textup{cost}(\widehat{A},\calB_{\rho},\varepsilon) 
	=\min\{ n : \lambda_{n+1} \le \varepsilon/\rho \} \\
	& = \min \left\{ n : \frac{1}{\lfloor (n+1)/2 \rfloor^r} \le \frac{\varepsilon}{\rho} \right\} 
	= 2 \left \lceil \left(\frac{\rho}{\varepsilon} \right)^{1/r} \right \rceil - 1.
	\end{align*}
	Here, $\calG=L^2[0,1]$.  The larger the non-negative parameter $r$ is, the faster the $\lambda_i$ tend to 0 as $ i \to 0$, the more exclusive $\calB_\rho$ is, and the smaller  $\textup{cost}(\widehat{A},\calB_{\rho},\varepsilon)$ is.  For $r = 0$, $\textup{cost}(\widehat{A},\calB_{\rho},\varepsilon) = \infty$.
\end{example}

Our goal is to construct algorithms in $\calA(\calH)$ for some $\calH$ and also to determine whether their computational cost is reasonable.  We define  $\textup{cost}(A,\calH,\varepsilon,\rho)$ to be \emph{essentially no worse} than $\textup{cost}(A^*,\calH^*,\varepsilon,\rho)$ if for some number $\omega$,
\begin{equation} \label{essNoWorseDef}
    \textup{cost}(A,\calH,\varepsilon,\rho)
    \le 
    \textup{cost}(A^*,\calH^*,\omega\varepsilon,\rho)
    \qquad \forall \varepsilon,\rho > 0.
\end{equation}
We extend this definition analogously if $\textup{cost}(A,\calH,\varepsilon,\rho)$ is replaced by $\textup{cost}(A,\calH,\varepsilon)$ and/or $\textup{cost}(A^*,\calH^*,\omega\varepsilon,\rho)$ is replaced by $\textup{cost}(A^*,\calH^*,\omega\varepsilon)$.
If these inequalities are not satisfied, we say that the cost of  $A$ is \emph{essentially worse}  than the cost of $A^*$.
If the costs of two algorithms are essentially no worse than each other, then we call them essentially the same.  An algorithm whose cost is essentially no worse than the best possible algorithm, is called \emph{essentially optimal}.

Our condition for essentially no worse cost in \eqref{essNoWorseDef} is not the same as 
\begin{equation} \label{NOTessNoWorseDef}
    \textup{cost}(A,\calH,\varepsilon,\rho)
    \le 
    \omega \, \textup{cost}(A^*,\calH^*,\varepsilon,\rho)
    \qquad \forall \varepsilon,\rho > 0.
\end{equation}
If the $\textup{cost}(A^*,\calH^*,\varepsilon,\rho)$ is proportional to $\max(1,\varepsilon^{-p})$, then conditions \eqref{essNoWorseDef} and \eqref{NOTessNoWorseDef} are equivalent.  If $\textup{cost}(A^*,\calH^*,\varepsilon,\rho)$ is proportional to $[1+\log(\max(1,\varepsilon)]^{-p}$, then \eqref{essNoWorseDef} is stricter than \eqref{NOTessNoWorseDef}.

To illustrate the comparison of costs, consider a non-increasing sequence of positive numbers, $ \{\lambda^*_1, \lambda^*_2, \ldots \}$, which converges to $0$, where $\lambda^*_i \ge \lambda_i$ for all $i \in \bbN$.  Also consider an unbounded strictly increasing sequence of non-negative integers $\bsn = \{n_0, n_1, \ldots \}$.  Define an algorithm $A^*$ analogously to $\widehat{A}$ defined in \eqref{optBallalg}:
\begin{equation*} \label{altBallalg}
A^*(f,\varepsilon) = A_{n_{j^\dagger}}(f), \quad \text{where } j^\dagger = \min\{ j : \lambda^*_{n_j+1} \le \varepsilon/\rho \}, \qquad A^* \in \calA(\calB_{\rho}).
\end{equation*}
By definition, the cost of algorithm $A^*$ is no better than that of  $\widehat{A}$.  Algorithm $A^*$ may or may not have essentially worse cost than $\widehat{A}$ depending on the choice of $\bslambda^*$ and $\bsn$.  The table below shows some examples.  Each different case of $A^*$ is labeled as having a cost that is either essentially no worse or essentially worse than that of $\widehat{A}$.
\everymath{\displaystyle}
\begin{longtable}{>{$}r<{$}@{\quad}>{$}l<{$}@{\qquad}>{$}l<{$}>{$}l<{$}>{$}l<{$}}
\toprule \toprule
&\lambda_i = \frac{C}{i^p}
& 
\textup{cost}(\widehat{A},\calB_\rho,\varepsilon) \ge \left ( \frac{C\rho}{\varepsilon} \right)^{1/p} - 1
\\[2ex]
& &
\textup{cost}(\widehat{A},\calB_\rho,\varepsilon) 
<  \left ( \frac{C\rho}{\varepsilon} \right)^{1/p}
\\[2ex]
\midrule
\text{no worse}
&
\lambda^*_i = \frac{C^*}{i^p}, \ n_j = 2^j
&
\textup{cost}(A^*,\calB_\rho,\varepsilon) \le 
2 \left ( \frac{C^*\rho}{\varepsilon} \right)^{1/p}
\\[2ex]
\midrule
\text{worse}
&
\lambda^*_i = \frac{C^*}{i^q}, \ q<p, \ n_j = j
&
\textup{cost}(A^*,\calB_\rho,\varepsilon)  \ge 
\left ( \frac{C^*\rho}{\varepsilon} \right)^{1/q} - 1 
\\[2ex]
\toprule \toprule
&\lambda_i = \frac{C}{p^i}, \ p > 1
& 
\textup{cost}(\widehat{A},\calB_\rho,\varepsilon) \ge \frac{\log (C\rho/\varepsilon)}{\log(p)} - 1
\\[2ex]
&&
\textup{cost}(\widehat{A},\calB_\rho,\varepsilon) < \frac{\log (C\rho/\varepsilon)}{\log(p)}
\\[2ex]
\midrule
\text{no worse}
&
\lambda^*_i = \frac{C^*}{p^i}, \ n_j = 2j
&
\textup{cost}(A^*,\calB_\rho,\varepsilon) < \frac{\log (C^*\rho/\varepsilon)}{\log(p)} + 1
\\[2ex]
\midrule
\text{worse}
&
\lambda^*_i = \frac{C^*}{p^i}, \ n_j = 2^j
&
\textup{cost}(A^*,\calB_\rho,\varepsilon) > 1.999 \frac{\log (C^*\rho/\varepsilon)}{\log(p)} 
\\[1ex]
&& \qquad \qquad \text{for some } \varepsilon
\\
\midrule
\text{worse}
&
\lambda^*_i = \frac{C^*}{i^q}, \ q<p, \ n_j = j
&
\textup{cost}(A^*,\calB_\rho,\varepsilon)  \ge 
\frac{\log (C^*\rho/\varepsilon)}{\log(q)} - 1
\\[2ex]
\toprule \toprule
\end{longtable}

\subsection{The Case for Adaptive Algorithms}
For bounded sets of input functions, such as balls, non-adaptive algorithms like $\widehat{A}$ make sense.  However, an a priori upper bound on $\DHJRnorm[\calF]{f}$ is typically unavailable in practice, so it is unknown which $\calB_{\rho}$ contain the input function $f$.  Our aim is to consider unbounded sets of $f$ for which the error of the interpolatory algorithm $A_n(f)$, defined in \eqref{optAdef}, can be bounded without an a priori upper bound on $\DHJRnorm[\calF]{f}$.

Adaptive algorithms encountered in practice typically employ heuristic error bounds.  While any algorithm can be fooled, we would like precise necessary conditions for being fooled, or equivalently, sufficient conditions for the algorithm to succeed.  Our adaptive algorithm has such conditions and follows in the vein of adaptive algorithms developed in \cite{HicEtal14a, HIcEtal14b, HicJim16a, JimHic16a}.

Our rigorous, data-driven error bound assumes the series coefficients of the input function, $f$, to decay steadily---but not necessarily monotonically.  The cone of nice input functions, $\calC$, is defined in Section \ref{sec:cone}.  For such inputs, we construct an adaptive algorithm, $\widetilde{A} \in \calA(\calC)$, in Section \ref{sec:adaptalgo},  where $\widetilde{A}(f,\varepsilon) = A_{n^*}(f)$ for some $n^*$ depending on the input data and the definition of $\calC$.  The number of series coefficients sampled, $n^*$, is adaptively determined so that $\widetilde{A}(f,\varepsilon)$ satisfies the error condition in \eqref{adapErrCrit}.  The computational cost of $\widetilde{A}$ is given in  Theorem \ref{thm:compcost}.  Section \ref{sec:opt} shows that our new algorithm is essentially optimal (see Theorem \ref{thm:CostNoWorse}).  Section \ref{sec:examp} provides an example of our algorithm.  We end with concluding remarks in Section \ref{sec:conc}.

\section{Assuming a Steady Decay of the Series Coefficients of the Solution} \label{sec:cone}

Recall from \eqref{errOpt} that the error of the fixed sample size interpolatory algorithm $A_n$ is $\DHJRnorm[\calG]{S(f) - A_n(f)} = \DHJRbignorm[2]{\bigl(\lambda_{i} \widehat{f}_{i} \bigr)_{i= n+1}^{\infty}}$.  The error depends on the series coefficients not yet observed, so at first it seems impossible to bound the error in terms of observed series coefficients.  

However, we can observe the partial sums 
\begin{equation} \label{sumdef}
\sigma_j(f) :
= \DHJRnorm[2]{ \left(\lambda_{i} \widehat{f}_{i} \right)_{i=n_{j-1}+1}^{n_j}}, \qquad j \in \bbN,
\end{equation}
where $\bsn  = \{n_0, n_1, \ldots\}$ is a strictly increasing, unbounded sequence of non-negative integers.  We define the cone of nice input functions to consist of those functions for which the $\sigma_j(f)$ decay at a given rate with respect to one another:
\begin{align} \label{decayconedef}
\calC &= \left\{ f \in \calF : \sigma_{j+r}(f) \le ab^r \sigma_j (f) \ \ \forall j,r \in \bbN \right\} \\
\nonumber
& = \left\{ f \in \calF : \sigma_j(f) \le \min_{1 \le r < j}\{ab^r\sigma_{j-r}(f)\} \ \ \forall j \in \bbN \right\}.
\end{align}
Here, $a$ and $b$ are positive numbers that define the inclusivity of the cone $\calC$ and satisfy
\begin{equation*} \label{abcond}
b <1 < a.
\end{equation*}
The constant $a$ is an inflation factor, and the constant $b$ defines the general rate of decay of the $\sigma_j(f)$ for $f \in \calC$. Because $ab^r$ may be greater than one, we do not require the series coefficients of the solution, $S(f)$, to decay monotonically. However, we expect their partial sums to decay steadily.

From the expression for the error in \eqref{errOpt} and the definition of the cone in  \eqref{decayconedef}, one can now derive a data-driven error bound for $j \in \bbN$:
\begin{align}
\nonumber
\DHJRnorm[\calG]{S(f)-A_{n_j}(f)} &= \DHJRnorm[2]{\left(\lambda_{i} \widehat{f}_{i} \right)_{i = n_j+1}^\infty} = \left\{\sum_{r=1}^\infty \sum_{i=n_{j+r-1}+1}^{n_{j+r}}  \DHJRabs{\lambda_{i}\widehat{f}_{i} }^{2}  \right\}^{1/2}\\
\nonumber
&= \DHJRnorm[2]{ \bigl(\sigma_{j+r}(f)\bigr)_{r=1}^{\infty}} \\
&\le \DHJRnorm[2]{ \bigl(ab^r\sigma_{j}(f)\bigr)_{r=1}^{\infty}}
 = \displaystyle ab \sqrt{\frac{1}{1 - b^2}}\sigma_{j}(f)
 \label{algoineq}
\end{align}
This upper bound depends only on the function data and the parameters defining $\calC$.  The error vanishes as $j \to \infty$ because $\sigma_j(f) \le ab^{j-1} \sigma_1(f) \to 0$ as $j \to \infty$.  Moreover, the error of $A_{n_j}(f)$ is asymptotically no worse than $\sigma_j(f)$, whose rate of decay need not be postulated in advance. Our adaptive algorithm in Section \ref{sec:adaptalgo} increases $j$ until the right hand side is smaller than the error tolerance.

Consider the choice 
\begin{equation*} \label{geonj}
n_j = 2^{j}n_0,
\end{equation*}
where the number of terms in the sums, $\sigma_j(f)$, are doubled at each step.  If the series coefficients of the solution decay like $\lambda_{i} \DHJRabs{f_{i}} = \calO(i^{-p})$ for some $p>1$, then it is reasonable to expect that the $\sigma_j(f)$ are bounded above and below as
\begin{equation} \label{algDecJ}
C_{\textup{lo}} (n_02^j)^{1-p} \le \sigma_j(f) \le C_{\textup{up}} (n_02^j)^{1-p}, \quad   j \in \bbN,
\end{equation}
for some constants $C_{\textup{lo}}$ and $C_{\textup{up}}$, unless the series coefficients drop precipitously in magnitude for some $n_{j-1} < i \le n_j$, and then jump back up for larger $i$.  When \eqref{algDecJ} holds, it follows that
\begin{equation*} 
\frac{\sigma_{j+r}(f)}{\sigma_j(f)} \le \frac{C_{\textup{up}} (n_02^{j+r})^{1-p}}{C_{\textup{lo}} (n_02^j)^{1-p}} = \frac{C_{\textup{up}} 2^{r(1-p)}}{C_{\textup{lo}}}\quad   j \in \bbN.
\end{equation*}
Thus, choosing $a \ge C_{\textup{up}}/C_{\textup{lo}}$ and $b \ge 2^{1-p}$ ensures that reasonable inputs $f$ lie inside the cone $\calC$.

\section{Adaptive Algorithm} \label{sec:adaptalgo}

Now we introduce our adaptive algorithm, $\widetilde{A} \in \calA(\calC)$, which yields an approximate solution to the problem $S:\calC\rightarrow\calG$ that meets the absolute error tolerance $\varepsilon$.

\begin{algo}\label{algo2}
Given $a$, $b$, the sequence $\bsn$, the cone $\calC$, the input function $f \in \calC$, and the absolute error tolerance $\varepsilon$, set $j=1$.
\begin{description}
\item[Step 1.] Compute $\sigma_{j}(f)$ as defined in \eqref{sumdef}.
\item[Step 2.] Check whether $j$ is large enough to satisfy the error tolerance, i.e.,
    \begin{equation*}\label{covcrit}
          \sigma_{j}(f) \le \frac{\varepsilon\sqrt{1 - b^2}}{ab} .
    \end{equation*}
    If this is true, then return $\widetilde{A}(f,\varepsilon) = A_{n_{j}}(f)$, where $A_n$ is defined in \eqref{optAdef}, and terminate the algorithm.
\item[Step 3.] Otherwise, increase $j$ by $1$ and return to Step $1$.
\end{description}
\end{algo}

\begin{theorem}\label{thm:compcost}
The algorithm, $\widetilde{A}$, defined in Algorithm \ref{algo2} lies in $\calA(\calC)$ and has computational cost $\textup{cost}(\widetilde{A},f,\varepsilon)=n_{j^*}$, where $j^*$ is defined implicitly by the inequalities 
\begin{equation} \label{eq:OurAlgjstar}
j^* = \min\left \{ j \in \bbN : \sigma_{j}(f) \le \frac{\varepsilon\sqrt{1 - b^2}}{ab}  \right\}.
\end{equation}
Moreover, $\textup{cost}(\widetilde{A},\calC,\varepsilon,\rho) \le n_{j^\dagger}$, where $j^\dagger$ satisfies the following upper bound:
\begin{equation} \label{jdagger}
j^\dagger \le \min \left \{j \in \bbN : \frac{\rho^2}{\varepsilon^2} \le \frac{(1 - b^2)}{a^2b^2} \left[ \sum_{k=1}^{j-1} \frac{b^{2(k-j)}}{a^2\lambda_{n_{k-1}+1}^2} + \frac{1}{\lambda_{n_{j-1}+1}^2}\right]   \right\}.
\end{equation}

\end{theorem}

\begin{proof}
This algorithm terminates for some $j = j^*$ because $\sigma_j(f) \le ab^{j-1} \sigma_{1}(f) \to 0$ as $j \to \infty$. The value of $j^*$ follows directly from this termination criterion in Step 2.  It then follows that the error bound on $A_{n_{j^*}}(f)$ in \eqref{algoineq} is no greater than the error tolerance $\varepsilon$.  So, $\widetilde{A} \in \calA(\calC)$.

For the remainder of the proof consider $\rho$ and $\varepsilon$ to be fixed.  To derive an upper bound on $n_{j^\dagger} = \textup{cost}(\widetilde{A},\calC,\varepsilon,\rho)$ we first note some properties of $\sigma_j(f)$ for all $f \in \calC$:
\begin{multline} \label{normsigineq}
\lambda_{n_j} \DHJRbignorm[2]{\big( \widehat{f}_i \big)_{i=n_{j-1}+1}^{n_j}} \le 
\DHJRbignorm[2]{\big( \lambda_i \widehat{f}_i \big)_{i=n_{j-1}+1}^{n_j}} = \sigma_j(f) 
\\
\le \lambda_{n_{j-1}+1} \DHJRbignorm[2]{\big( \widehat{f}_i \big)_{i=n_{j-1}+1}^{n_j}}.
\end{multline}

A rough upper bound on $j^\dagger$ may be obtained by noting that for any $f \in \calC \cap \calB_\rho$ and for any $j < j^* \le j^\dagger$, it follows from \eqref{eq:OurAlgjstar} and \eqref{normsigineq} that 
\begin{equation*}
\rho \ge \DHJRnorm[\calF]{f} \ge \DHJRbignorm[2]{\big( \widehat{f}_i \big)_{i=n_{j-1}+1}^{n_j}} \ge \frac{\sigma_j(f)}{\lambda_{n_{j-1}+1}} > \frac{\varepsilon\sqrt{1 - b^2}}{ab \lambda_{n_{j-1}+1}}
\end{equation*}
Thus, one upper bound on $j^\dagger$ is the smallest $j$ violating the above inequality:
\begin{equation} \label{jdaggerlast}
j^\dagger \le \min \left \{j \in \bbN :  \lambda_{n_{j-1}+1} \le \frac{\varepsilon\sqrt{1 - b^2}}{ab \rho} \right\}.
\end{equation}

The tighter upper bound in Theorem \ref{thm:compcost} may be obtained by a more careful argument in a similar vein.  
For any $f \in  \calC \cap \calB_\rho$ and for any $j < j^* \le j^\dagger$,
\begin{align*}
\rho^2 &\ge \DHJRnorm[\calF]{f}^2 = \DHJRbignorm[2]{\big( \widehat{f}_i \big)_{i=1}^{\infty}}^2 \\
& \ge \sum_{k=1}^j \DHJRbignorm[2]{\big( \widehat{f}_i \big)_{i=n_{k-1}+1}^{n_k}}^2 \qquad \forall j \ge 1\\
& \ge \sum_{k=1}^j \frac{\sigma_k^2(f)}{\lambda_{n_{k-1}+1}^2} \qquad \text{by \eqref{normsigineq}}\\
& \ge \sum_{k=1}^{j-1} \frac{ b^{2(k-j)}\sigma_j^2(f)}{a^{2}\lambda_{n_{k-1}+1}^2} + \frac{\sigma_j^2(f)}{\lambda_{n_{j-1}+1}^2} \qquad \text{by \eqref{decayconedef}} \\
& = \sigma_j^2(f) \left[ \sum_{k=1}^{j-1} \frac{b^{2(k-j)}}{a^{2} \lambda_{n_{k-1}+1}^2} + \frac{1}{\lambda_{n_{j-1}+1}^2}\right].
\end{align*}
Note that the quantity in the square brackets is an increasing function of $j$ because as $j$ increases, the sum includes more terms and $b^{2(k-j)}$ also increases.

For all $j < j^* \le j^\dagger$ it follows from \eqref{eq:OurAlgjstar} that
\begin{equation*}
\rho^2 > \frac{\varepsilon^2(1 - b^2)}{a^2b^2} \left[ \sum_{k=1}^{j-1} \frac{ b^{2(k-j)}}{a^{2}\lambda_{n_{k-1}+1}^2} + \frac{1}{\lambda_{n_{j-1}+1}^2}\right].
\end{equation*}
Thus, any $j$ that violates the above inequality, must satisfy $j \ge j^\dagger$, establishing \eqref{jdagger}.
\end{proof}

We note in passing that for our adaptive algorithm
\begin{equation*}
 \min \{\textup{cost}(\widetilde{A},f,\varepsilon) : f \in \calC, \ \DHJRnorm[\calF]{f} \ge \rho \} 
 \begin{cases} = n_1, & n_0 > 0, \\
 \le n_2, & n_0 = 0, 
 \end{cases}
 \qquad \forall \rho > 0, \ \varepsilon > 0.
\end{equation*}
This result may be obtained by considering functions where only $\widehat{f}_1$ is nonzero.  For $n_0 > 0$, $\sigma_1(f) = 0$, and for $n_0 = 0$, $\sigma_2(f) = 0$.

The upper bound on $\textup{cost}(\widetilde{A},\calC,\rho,\varepsilon)$ in Theorem \ref{thm:compcost}  is a non-decreasing function of $\rho/\varepsilon$, which depends on the behavior of the sequence $\{(\lambda_{n_j})_{j=0}^\infty\}$.  This in turn depends both on the increasing sequence $\bsn$ and on the non-increasing sequence $\{(\lambda_i)_{i=1}^\infty\}$. Consider the  term enclosed in square brackets on the the right hand side of the inequality in \eqref{jdagger}: \begin{equation*} \label{keysum}
\sum_{k=1}^{j-1} \frac{b^{2(k-j)}}{a^2\lambda_{n_{k-1}+1}^2} + \frac{1}{\lambda_{n_{j-1}+1}^2}.
\end{equation*}
One can imagine that in some cases the first term in the sum dominates, while in other cases the term outside the sum dominates, all depending on how $b^{k-j}/\lambda_{n_{k-1}+1}$ behaves with $k$ and $j$.  These simplifications lead to two simpler, but coarser upper bounds on the cost of $\widetilde{A}$.

\begin{corollary} For the algorithm, $\widetilde{A}$, defined in Algorithm \ref{algo2}, then $\textup{cost}(\widetilde{A},\calC,\varepsilon,\rho) \le n_{j^\dagger}$, where $j^\dagger$ satisfies the following upper bound:
\begin{equation} \label{jdaggerfirst}
j^\dagger \le \left \lceil \log\left(\frac{\rho a^2\lambda_{n_{0}+1} }{\varepsilon \sqrt{1 - b^2}}\right) \div \log\left(\frac{1}{b}\right) \right \rceil.
\end{equation}
Moreover, if the $\lambda_{n_{j-1}+1}$ decay as quickly as
\begin{equation}
\label{lambdankbd}
\lambda_{n_{j-1}+1} \le \alpha \beta^j,  \quad j \in \bbN,  \qquad \text{for some } \alpha > 0, \ 0 < \beta < 1.
\end{equation}
then $j^\dagger$ also satisfies the following upper bound:
\begin{equation}
\label{jdaggerlastsimple}
j^\dagger \le
\left \lceil \log\left(\frac{\rho a\alpha b }{\varepsilon \sqrt{1 - b^2}}\right) \div \log\left(\frac{1}{\beta}\right) \right \rceil.
\end{equation}
\end{corollary}

\begin{proof}
Ignoring all but the first term in the sum in \eqref{keysum} implies that 
\begin{equation*} \label{jsumupperbd}
j^\dagger \le \min \left \{j \in \bbN : \frac{\rho^2}{\varepsilon^2} \le \frac{(1 - b^2)}{a^2b^2} \frac{b^{2(1-j)}}{a^2\lambda_{n_{0}+1}^2}    \right\}.
\end{equation*}
This implies \eqref{jdaggerfirst}.

Ignoring all but the term outside the sum leads to the simpler upper bound in \eqref{jdaggerlast}.  If the $\lambda_{n_{j-1}+1}$ decay as assumed in \eqref{lambdankbd} then
\[
j^\dagger \le \min \left \{j \in \bbN :  \alpha \beta^j \le \frac{\varepsilon\sqrt{1 - b^2}}{ab \rho} \right\},
\]
which implies \eqref{jdaggerlastsimple}.
\end{proof}

This corollary highlights two limiting factors on the computational cost of our adaptive algorithm, $\widetilde{A}$. When $j$ is large enough to make $\lambda_{n_{j-1}+1}\DHJRnorm[\calF]{f}/\varepsilon$ small enough, $\widetilde{A}(f,\varepsilon)$ stops.  This is statement \eqref{jdaggerlastsimple}, and its precursor, \eqref{jdaggerlast}.  Alternatively, the assumption that the $\sigma_j(f)$ are steadily decreasing, as specified in the definition of $\calC$ in \eqref{decayconedef}, means that $\widetilde{A}(f,\varepsilon)$ also must stop by the time $j$ becomes large enough with respect to $\lambda_{n_0+1}\DHJRnorm[\calF]{f}/\varepsilon$.

Assumption \eqref{lambdankbd} is not very restrictive.  It holds if the $\lambda_i$ decay algebraically and the $n_j$ increase geometrically.  It also holds if the $\lambda_i$ decay geometrically and the $n_j$ increase arithmetically.

The adaptive algorithm $\widetilde{A}$, which does not know an upper bound on $\DHJRnorm[\calF]{f}$ a priori, may cost more than the non-adaptive algorithm $\widehat{A}$, which assumes an upper bound on $\DHJRnorm[\calF]{f}$, but under reasonable assumptions, the extra cost is small.

\begin{corollary} \label{cor:tAsameCosthA} Suppose that the sequence $\bsn$ is chosen to satisfy
\begin{equation} \label{lambdaDecay}
\lambda_{n_{j+1}+1} \ge c_\lambda \lambda_{n_j+1}, \qquad j \in \bbN, 
\end{equation}
for some positive $c_\lambda$.  Then $\textup{cost}(\widetilde{A},\calC, \varepsilon,\rho)$ is essentially no worse than \linebreak[4]
$\textup{cost}(\widehat{A},\calB_{\rho},\varepsilon)$ in the sense of \eqref{essNoWorseDef}.

\end{corollary}

\begin{proof}
Combining the upper bound on $n_{j^\dagger} = \textup{cost}(\widetilde{A},\calC,\varepsilon,\rho)$ in \eqref{jdaggerlast} plus  \eqref{lambdaDecay} above, it follows that
\begin{equation*}
\lambda_{n_{j^\dagger}+1} \ge c_{\lambda}^2 \lambda_{n_{j^\dagger-2}+1} > \frac{\varepsilon c_{\lambda}^2\sqrt{1 - b^2}}{ab \rho} \ge \lambda_{n+1},
\end{equation*}
where $n = \textup{cost}(\widehat{A},\calB_{\rho},\varepsilon c_{\lambda}^2\sqrt{1 - b^2}/ab)$.
Since the $\lambda_i$ are non-increasing,
\begin{equation*}
\textup{cost}(\widetilde{A},\calC,\varepsilon,\rho) = n_{j^\dagger} \le n_{j^\dagger}+1 < n = \textup{cost}(\widehat{A},\calB_{\rho},\varepsilon c_{\lambda}^2\sqrt{1 - b^2}/ab).
\end{equation*}
\end{proof}

\section{Essential Optimality of the Adaptive Algorithm} \label{sec:opt}

From Corollary \ref{cor:tAsameCosthA} it is known that $\textup{cost}(\widetilde{A},\calC, \varepsilon,\rho)$ is essentially no worse than
$\textup{cost}(\widehat{A},\calB_{\rho},\varepsilon) = \textup{comp}(\calA(\calB_{\rho}),\varepsilon)$.  We would like to show that $\widetilde{A} \in \calA(\calC)$ is  essentially optimal, i.e., $\textup{cost}(\widetilde{A},\calC, \varepsilon,\rho)$ is essentially no worse than  $\textup{comp}(\calA(\calC),\varepsilon,\rho)$.  However,  $\textup{comp}(\calA(\calC),\varepsilon,\rho)$ may be smaller than $\textup{comp}(\calA(\calB_{\rho}),\varepsilon)$ because $\calC \cap \calB_\rho$ is a strict subset of  $ \calB_\rho$.  This presents a challenge.

A lower bound on $\textup{comp}(\calA(\calC),\varepsilon,\rho)$ is established by constructing fooling functions in $\calC$ with norms no greater than $\rho$.  To obtain a result that can be compared with the cost of our algorithm, we assume that 
\begin{equation} \label{lambdaRatio}
R = \sup_{k \in \bbN} \frac{\lambda_{n_{k-1}}}{\lambda_{n_k}} < \infty.
\end{equation}
This means that the $n_k$ are not too far apart with respect to the decay of $\lambda_i$ as $i \to \infty$.

The following theorem establishes a lower bound on the complexity of our problem for input functions in $\calC$.  The theorem after that shows that the cost of our algorithm as given in Theorem \ref{thm:compcost} is essentially no worse than this lower bound.

\begin{theorem} \label{thm:lowbdcomp}
Under assumption  \eqref{lambdaRatio}, a lower bound on the complexity of the linear problem defined in \eqref{probDef} is
\begin{align*}
&\textup{comp}(\calA(\calC),\varepsilon,\rho) \ge n_{j^*}, 
\intertext{where}
j^* & = \max \left \{ j \in \bbN : \left[\frac{(a+1)^{2} R^2 }{(a-1)^2} + 1\right] \sum_{k=0}^j \frac{b^{2(k-j)}}{\lambda_{n_{k}}^2}   <
\frac{\rho^2}{\varepsilon^2}
\right \}.
\end{align*}
\end{theorem}

\begin{proof}
Consider a fixed $\rho$ and $\varepsilon$.  Choose any positive integer $j$ such that $n_j$ exceeds $\textup{comp}(\calA(\calC),\varepsilon,\rho)$.  The proof proceeds by carefully constructing three test input functions, $f$ and $f_{\pm}$, lying in $\calC \cap \calB_{\rho}$, which yield the same approximate solution but different true solutions.  This leads to a lower bound on $n_j$, which can be translated into a lower bound on $\textup{comp}(\calA(\calC),\varepsilon,\rho)$. 

The first test function $f \in \calC$ is defined in terms of its series coefficients:
\begin{align}
\nonumber
\widehat{f}_i &:= \begin{cases}
\displaystyle
\frac{c b^{k-j}}{\lambda_{n_{k}}},  & i =  n_{k}, \ k = 1, \ldots, j,
\\
0, & \text{otherwise},
\end{cases}
\\
\nonumber
c^2 &:=  \rho^2 \left[ \left(1+\frac{(a-1)^2}{(a+1)^{2} R^2 }\right)\sum_{k=0}^j \frac{b^{2(k-j)}}{\lambda_{n_{k}}^2}\right]^{-1}.
\end{align}
It can be verified that the test function lies both in $\calB_{\rho}$ and in $\calC$:
\begin{align}
\nonumber
\DHJRnorm[\calF]{f}^2 &=  c^2\sum_{k=1}^j \frac{ b^{2(k-j)}}{\lambda^2_{n_{k}}} \le \rho^2,
\\
\nonumber
\sigma_k(f) &= \begin{cases}
\displaystyle
c b^{k-j}, & k =1, \ldots, j, \\
0, & \text{otherwise},
\end{cases}
\\
\nonumber
\sigma_{k+r}(f) &= \begin{cases}
\displaystyle 
b^{r} \sigma_k(f) \le a b^r \sigma_k(f), & k+r \le j, \ r \ge 1,
\\
0 \le a b^r \sigma_k(f), & k+r > j, \ r \ge 1.
\end{cases}
\end{align}

Now suppose that $A^* \in \calA(\calC)$ is an optimal algorithm, i.e., $\textup{cost}(A^*,\calC,\varepsilon,\rho) =  \textup{comp}(\calA(\calC),\varepsilon,\rho)$ for all $\varepsilon, \rho > 0$.  For our particular input $f$ defined above, suppose that $A^*(f,\varepsilon)$ samples $L_1(f), \ldots, L_n(f)$ where 
\[
n + 1\le \textup{comp}(\calA(\calC),\varepsilon,\rho) +1 < n_j.
\]  

Let $u$ be a linear combination of $u_1, \cdots, u_{n_j}$, expressed as
\[
u =  \sum_{k=0}^{j}\frac{b^{k-j}u^{(k)}}{\lambda_{n_k}},
\]
where $u^{(0)}$ is a linear combination of $u_{1}, \ldots, u_{n_0}$, and each $u^{(k)}$ is a linear combination of $u_{n_{k-1}+1}, \ldots, u_{n_k}$, for $k =1, \ldots, j$.  We constrain $u$ to satisfy:
\begin{equation}\label{uConstraint}
L_1(u) = \cdots = L_n(u) = 0, \qquad \langle u,f \rangle_{\calF} = 0, \qquad 
\max_{0\le k\le j} \DHJRbignorm[\calF]{u^{(k)}} = 1.
\end{equation}
Since $u$ is a linear combination of $n_j >n+1$ basis functions, these constraints can be satisfied.

Let the other two test functions be constructed in terms of $u$ as 
\begin{align}
\nonumber
f_\pm & := f \pm \eta u, \qquad \eta : =  \frac {(a-1)c}{(a+1)R}, \\
\nonumber
\DHJRnorm[\calF]{f_\pm}^2 & \le \DHJRnorm[\calF]{f}^2 + \DHJRnorm[\calF]{\eta u }^2 \qquad \text{by \eqref{uConstraint}} \\
& \nonumber 
\le \sum_{k=1}^j \frac{ b^{2(k-j)}}{\lambda^2_{n_{k}}} \left(c^2+ \eta^2 \DHJRbignorm[\calF]{u^{(k)}}^2\right) + \eta^2 \DHJRbignorm[\calF]{u^{(0)}}^2 \frac{b^{-2j}}{\lambda_{n_0}^2}\\
\nonumber
& \le  \left(c^2+ \eta^2 \right) \sum_{k=0}^j \frac{ b^{2(k-j)}}{\lambda^2_{n_{k}}}  \qquad \text{by \eqref{uConstraint}}\\
& \nonumber 
\le \rho^2,
\end{align} 
so $f_{\pm} \in \calB_{\rho}$.  By design, $A^*(f_\pm,\varepsilon) = A^*(f,\varepsilon)$, which will be used below.

Now we must check that $f_\pm \in \calC$. From the definition in \eqref{sumdef} it follows that for $k = 1, \ldots, j$ and $r \ge 1$,
\begin{equation*}
\sigma_k(f_\pm)  \begin{cases} 
\displaystyle
\le \sigma_k(f) + \sigma_k(\eta u)\le 
c b^{k-j} + \eta\lambda_{n_{k-1}+1}\frac{b^{k-j}\DHJRnorm[\calF]{u^{(k)} } }{\lambda_{n_k}}
\le b^{k-j}\left(c+\eta R \right) 
\\[1ex]
\displaystyle
\ge \sigma_k(f) - \sigma_k(\eta u)\ge 
c b^{k-j} - \eta\lambda_{n_{k-1}+1}\frac{b^{k-j}\DHJRnorm[\calF]{u^{(k)} } }{\lambda_{n_k}}
\ge b^{k-j}\left(c-\eta R \right) , 
\end{cases}
\end{equation*}
Therefore, 
\begin{equation*}
\sigma_{k+r}(f_\pm)
\le b^{k+r-j}(c+\eta R) = ab^r b^{k-j}\frac{2c}{a+1}
=ab^r b^{k-j}\left(c-\eta R \right) \le a b^r \sigma_{k}(f_\pm),
\end{equation*}
which establishes that $f_\pm \in \calC$.

Although two test functions $f_\pm$ yield the same approximate solution, they have different true solutions.  In particular,
\begin{align*}
\varepsilon &\ge \max \bigl\{\DHJRnorm[\calG]{S(f_+) - A^*(f_+,\varepsilon)}, \DHJRnorm[\calG]{S(f_-) - A^*(f_-,\varepsilon)} \bigr\} \\
&\ge \frac 12 \bigl[\DHJRnorm[\calG]{S(f_+) - A^*(f,\varepsilon)} + \DHJRnorm[\calG]{S(f_-) - A^*(f,\varepsilon)}  \bigr] \\
&\qquad \qquad \text{since } A^*(f_\pm,\varepsilon) = A^*(f,\varepsilon) \\
&\ge \frac 12 \DHJRnorm[\calG]{S(f_+) - S(f_-)} \quad \text{by the triangle inequality}\\
&\ge \frac 12 \DHJRnorm[\calG]{S(f_+ - f_-)} \quad \text{since $S$ is linear}\\
&= \eta \DHJRnorm[\calG]{S(u)}.
\end{align*}
Thus, we have
\begin{align*}
\varepsilon^2  & \ge \eta^2 \DHJRnorm[\calG]{S(u)}^2= 
\eta^2 \sum_{k=0}^{j} \DHJRbignorm[\calG]{S(u^{(k)})}^2  \frac{b^{2(k-j)}}{\lambda_{n_k}^2}\\
& \ge \eta^2 
\sum_{k=0}^{j} \DHJRbignorm[\calF]{u^{(k)}}^2 
b^{2(k-j)}\\
& \ge  \eta^2 b^{2(k^{*}-j)} \qquad \text{ where } k^* = \argmax_{0 \le k \le j} \DHJRbignorm[\calF]{u^{(k)}} \\
& \ge \eta^2 = \frac{(a-1)^2c^2}{(a+1)^2R^2} \\
&
=\frac{(a-1)^2 \rho^2}{(a+1)^2R^2}\left[\left(1+\frac{(a-1)^2}{(a+1)^{2} R^2 }\right)\sum_{k=0}^j \frac{b^{2(k-j)}}{\lambda_{n_{k}}^2}\right]^{-1} \\
& =  \rho^2 \left[\left\{\frac{(a+1)^{2} R^2 }{(a-1)^2} + 1\right\}\sum_{k=0}^j \frac{b^{2(k-j)}}{\lambda_{n_{k}}^2}\right]^{-1} 
\end{align*}

This lower bound must be satisfied by $j$ to be consistent with the assumption $\textup{comp}(\calA(\calC),\varepsilon,\rho) \le n_j - 1$.  Thus, for any $j$ violating this inequality it follows that $\textup{comp}(\calA(\calC),\varepsilon,\rho) \ge n_j$.  This implication provides a lower bound on $\textup{comp}(\calA(\calC),\varepsilon,\rho)$.
\end{proof}

The next step is to show that the cost of our algorithm is essentially no worse than that of the optimal algorithm.

\begin{theorem}
\label{thm:CostNoWorse}
Under assumption \eqref{lambdaRatio} $\textup{cost}(\widetilde{A},\calC, \varepsilon,\rho)$ is essentially no worse than $\textup{comp}(\calA(\calC),\varepsilon,\rho)$. 
\end{theorem}
\begin{proof}
Let  
\begin{equation} \label{omegadef}
    \omega = \sqrt{\frac{(1 - b^2)}{a^4(1 + b^2R^2 + b^4R^4)}\left[\frac{(a+1)^{2} R^2 }{(a-1)^2} + 1\right]^{-1}},
\end{equation}
and note that it does not depend on $\rho$ or $\varepsilon$ but only on the definition of $\calC$. 
For any positive $\rho$ and $\varepsilon$, Theorem \ref{thm:compcost} says that $\textup{cost}(\widetilde{A},\calC,\varepsilon,\rho) \le n_{j^\dagger}$, where 
\begin{align*} 
j^\dagger &\le \min \left \{j \in \bbN : \frac{\rho^2}{\varepsilon^2} \le \frac{(1 - b^2)}{a^2b^2} \left[ \sum_{k=1}^{j-1} \frac{b^{2(k-j)}}{a^2\lambda_{n_{k-1}+1}^2} + \frac{1}{\lambda_{n_{j-1}+1}^2}\right]   \right\} \\
&\le \min \left \{j \in \bbN : \frac{\rho^2}{\varepsilon^2} \le \frac{(1 - b^2)}{a^4b^2} \sum_{k=1}^{j} \frac{b^{2(k-j)}}{\lambda_{n_{k-1}+1}^2} \right\} \qquad \text{since } a > 1\\
&\le \min \left \{j \in \bbN : \frac{\rho^2}{\varepsilon^2} \le \frac{(1 - b^2)}{a^4} \sum_{k=0}^{j-1} \frac{b^{2(k-j)}}{\lambda_{n_{k}+1}^2} \right\} \\
&\le \min \left \{j \in \bbN : \frac{\rho^2}{\varepsilon^2} \le \frac{(1 - b^2)}{a^4} \sum_{k=0}^{j-1} \frac{b^{2(k-j)}}{\lambda_{n_{k}}^2} \right\} \qquad \text{since } \lambda_{n_k} \ge \lambda_{n_k+1}\\
&\le \min \left \{j \in \bbN : \frac{\rho^2}{\varepsilon^2} \le \frac{(1 - b^2)}{a^4(1 + b^2R^2 + b^4R^4)} \sum_{k=0}^{j+1} \frac{b^{2(k-j)}}{\lambda_{n_{k}}^2} \right\} \qquad \text{by \eqref{lambdaRatio} } \\
&\le \min \left \{j \in \bbN : \frac{\rho^2}{\omega^2 \varepsilon^2} \le \left[\frac{(a+1)^{2} R^2 }{(a-1)^2} + 1\right] \sum_{k=0}^{j+1} \frac{b^{2(k-j)}}{\lambda_{n_{k}}^2} \right\} \qquad \text{by \eqref{omegadef} } \\
&= \max \left \{j \in \bbN : \frac{\rho^2}{\omega^2 \varepsilon^2} > \left[\frac{(a+1)^{2} R^2 }{(a-1)^2} + 1\right] \sum_{k=0}^{j} \frac{b^{2(k-j)}}{\lambda_{n_{k}}^2} \right\} =:j^*.
\end{align*}
By Theorem \ref{thm:lowbdcomp}, $\textup{comp}(\calA(\calC),\omega\varepsilon,\rho) \ge n_{j^*}$, and by the argument above, $n_{j^*} \ge n_{j^\dagger} \ge \textup{cost}(\widetilde{A},\calC,\varepsilon,\rho)$.  Thus, our algorithm is essentially no more costly than the optimal algorithm.
\end{proof}

\section{Numerical Example} \label{sec:examp}

Consider the case of approximating the  partial derivative with respect to $x_1$ of periodic functions defined on the $d$-dimensional unit cube:
\begin{align*}
\allowdisplaybreaks
f &= \sum_{\bsk \in \bbZ^d} \widehat{f}(\bsk) \widehat{u}_{\bsk} = \sum_{i \in \bbN} \widehat{f}_i u_i , \\
\widehat{u}_{\bsk} (\bsx) &:= \prod_{j=1}^d \frac{2^{(1-\delta_{k_j,0})/2}\cos(2\pi k_j x_j  + \bbone_{(-\infty,0)}(k_j) \pi/2 )}{\max^4(1,\gamma_j k_j)},  \\ 
S(f) &: = \frac{\partial f}{\partial x_1} = \sum_{\bsk \in \bbZ^d} \widehat{f}(\bsk) \lambda(\bsk) \widehat{v}_{\bsk} (\bsx)
= \sum_{i \in \bbN} \widehat{f}_i \lambda_i v_i, \\ 
\widehat{v}_{\bsk}(\bsx) & : =  -\textup{sign}(k_1)  \sin(2\pi k_1 x_1  + \bbone_{(-\infty,0)}(k_1) \pi/2 ) \\
& \qquad \qquad \times \prod_{j=2}^d
\cos(2\pi k_j x_j  + \bbone_{(-\infty,0)}(k_j) \pi/2 ), \\ 
\lambda(\bsk) &:= 2 \pi \DHJRabs{k_1}\frac{\prod_{j=1}^d 2^{(1-\delta_{k_j,0})/2}}{\prod_{j=1}^d\max^4(1,\gamma_j k_j)}, \\
\boldsymbol{\gamma} & := (1,1/2, 1/4, \ldots, 2^{-d+1}).
\end{align*}
Note that $\lambda_1 \ge \lambda_2 \ge \cdots$ is an ordering of the $\lambda(\bsk)$.  That ordering then determines the $\widehat{f}_i, u_i$, and $v_i$ in terms of the $\widehat{f}(\bsk), \widehat{u}(\bsk)$, and $\widehat{v}(\bsk)$, respectively.

We construct a function by choosing its Fourier coefficients $\widehat{f}(\bsk) \overset{\text{IID}}{\sim}  \mathcal{N}(0,1)$ for $d=3$, $\bsk \in \{-30, -29, \ldots, 30\}^3$, and $\widehat{f}(\bsk) = 0$ otherwise.  This corresponds to $61^3 \approx 2 \times 10^5$ nonzero Fourier coefficients.  Let $ a= 2$ and  $b=1/2$ and 
 choose $\bsn = \{0, 16, 32, 64, \ldots \}.$
 To  compute $\sigma_j(f), \ j \in \bbN$ by \eqref{sumdef},
 we need to sort $\bigl(\lambda(\bsk)\bigr)_{\bsk \in \bbZ^d}$ in descending order, $\lambda_1, \lambda_2, \ldots$. Given $\varepsilon$, we can then find the number of series coefficients needed to satisfy the the error criterion, i.e., $n_{j^\dagger}$ where
 \[j^{\dagger} = \min \left\{j \in \bbN : \frac{ab\sigma_j(f)}{\sqrt{1-b^2}}  \le \varepsilon. \right\}\]
 
Fig.\ \ref{solfig} shows the input function, the solution, the approximate solution, and the error of the approximate solution for $\varepsilon = 0.1$.  For this example, $ n_{j^\dagger} = 8192$ is sufficient to satisfy the error tolerance, as is clear from Fig.\ \ref{solfig}(d).  Fig. \ref{errfig} shows the sample size, $n_{j^\dagger}$ needed for ten different error tolerances from $0.1$ to $10$. Because the possible sample sizes are powers of $2$ , some tolerances require the same sample size.

\begin{figure}[ht]
	\centering
	\begin{tabular}{cc}
		\includegraphics[width =5.5 cm]{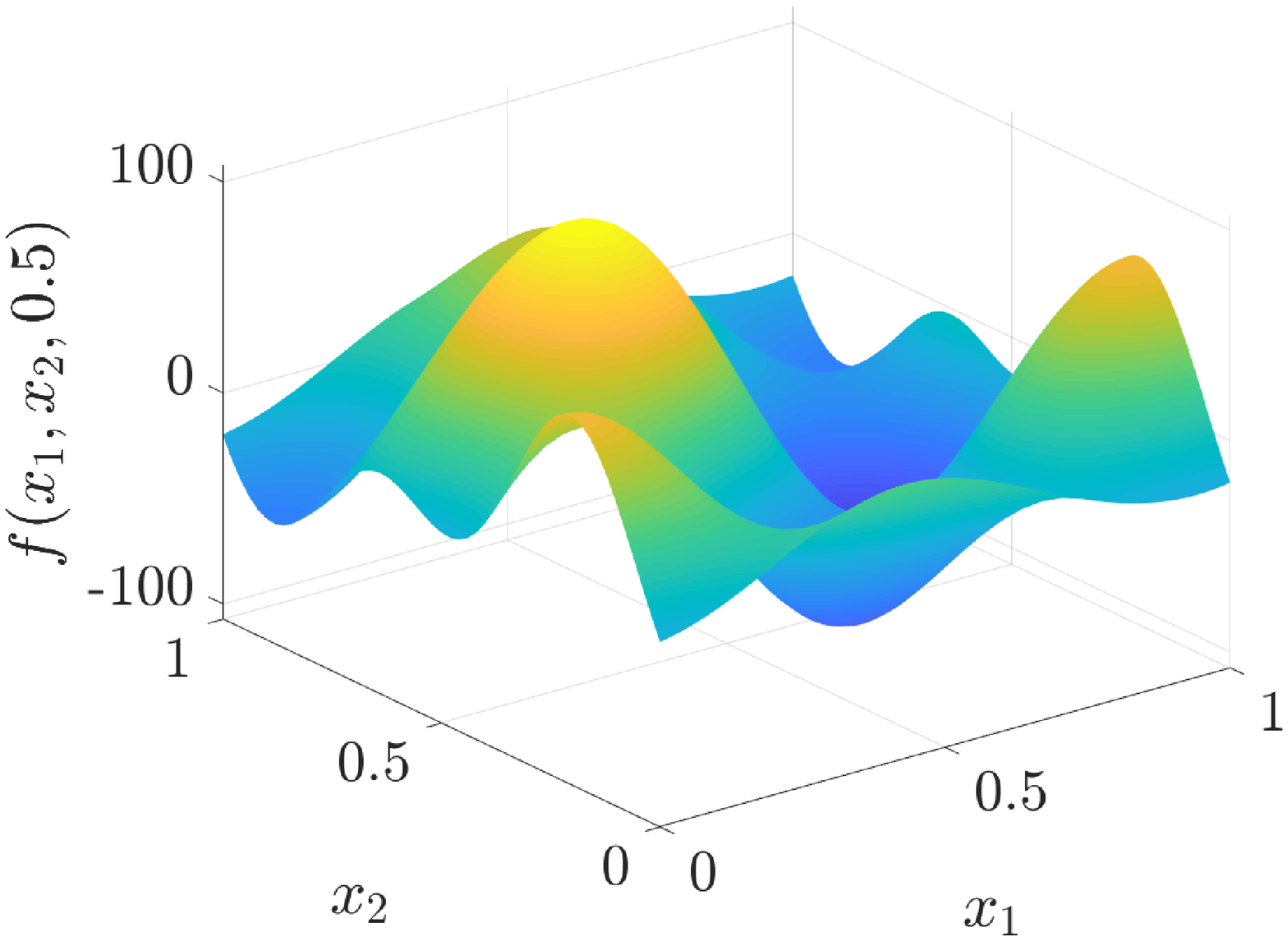} &
		\includegraphics[width = 5.5 cm]{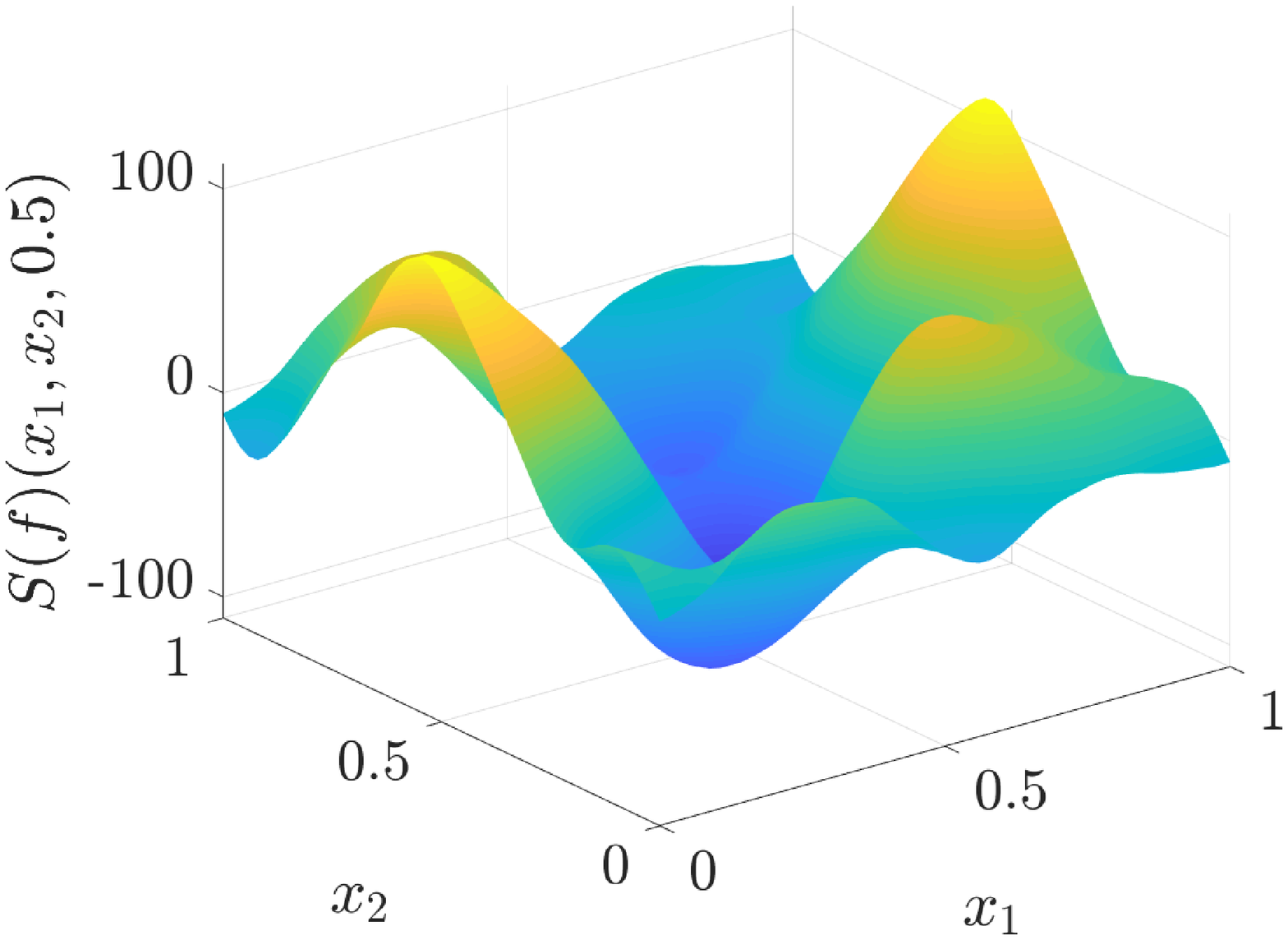}
		\\ (a) & (b) \\
		\includegraphics[width =5.5 cm]{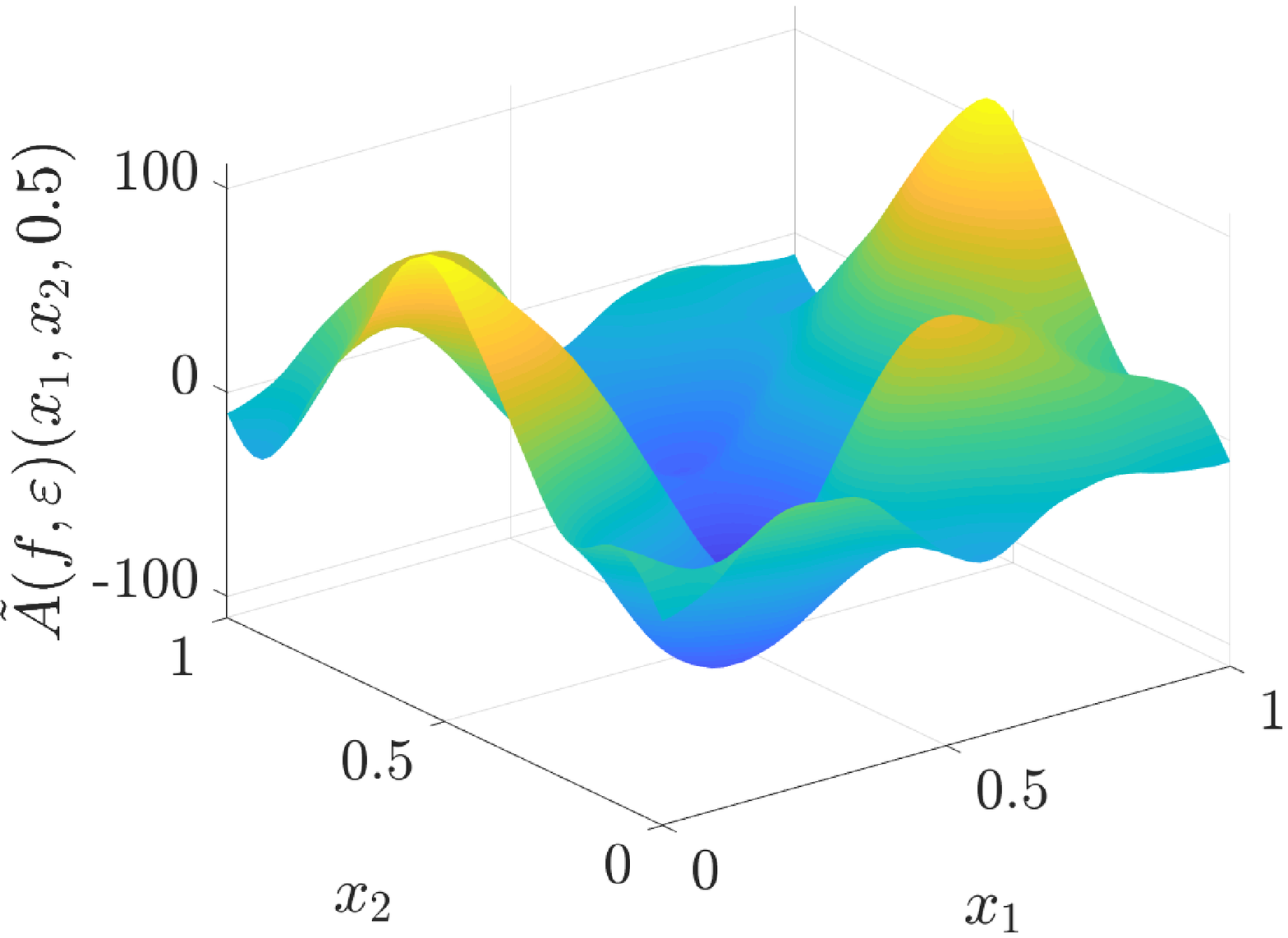} &
		\includegraphics[width = 5.5 cm]{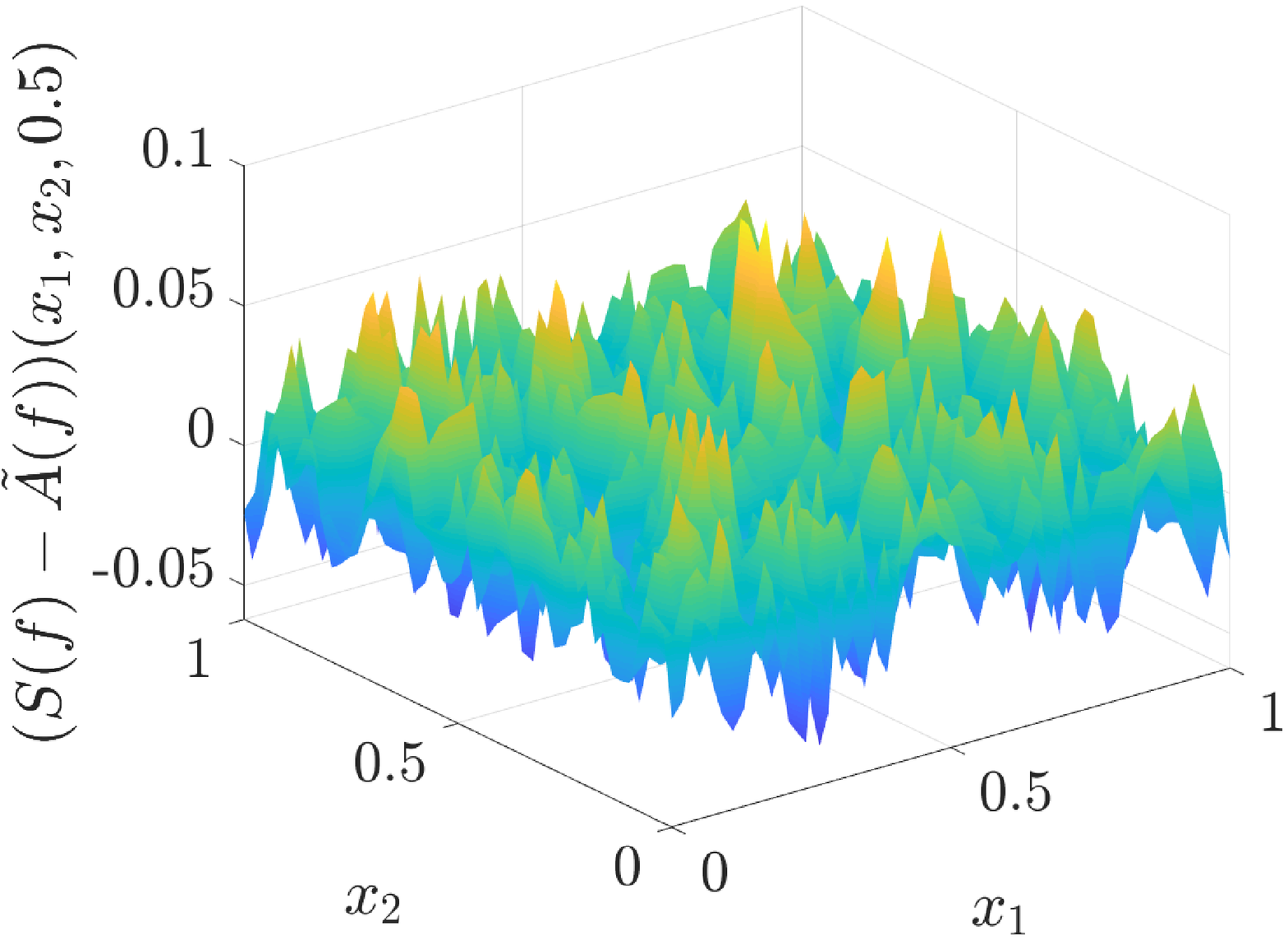}
		\\ (c) & (d)
	\end{tabular}
	\caption{For $\varepsilon = 0.1$:  (a) The input function, $f$; (b) The true first partial derivative of $f$; (c) The approximate first partial derivative of $f$; 
		(d) The approximation error.
		\label{solfig}} %
\end{figure}

\begin{figure}[ht]
 	\centering
 		\includegraphics[width =7.5 cm]{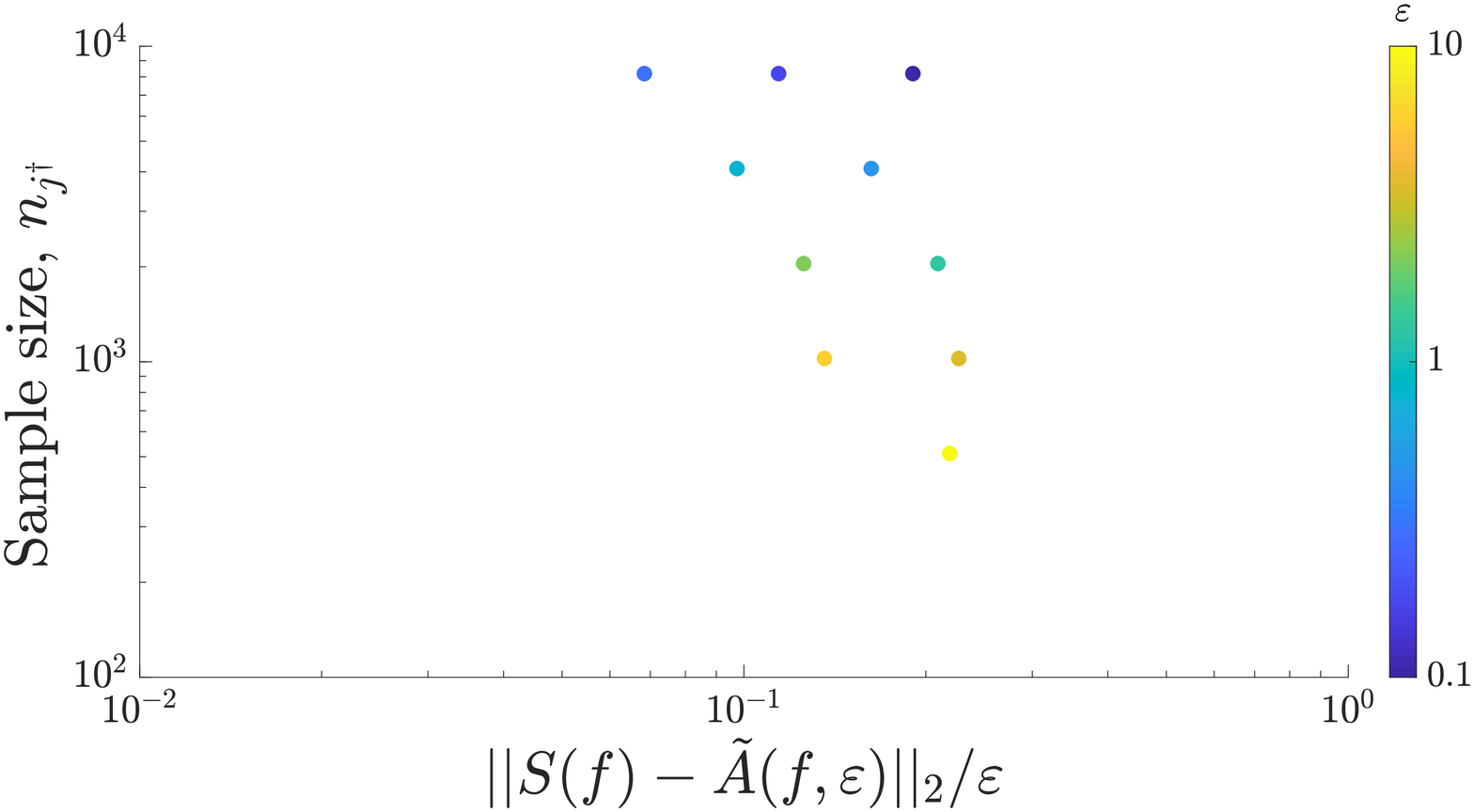}
 	\caption{Sample size $n_{j^{\dagger}}$, error tolerance $\varepsilon$, and ratio of true error to error tolerance. \label{errfig}} %
 \end{figure}

\section{Discussion and Conclusion} \label{sec:conc}
Many practical adaptive algorithms lack theory, and many theoretically justified algorithms are non-adaptive.  We have demonstrated for a general setting how to construct a theoretically justified, essentially optimal algorithm.  The decay of the singular values determines the computational complexity of the problem and the computational cost of our algorithm.  

The key idea of our algorithm is to derive an adaptive error bound by assuming the steady decay of the Fourier series coefficients of the solution.  The set of such functions constitutes a cone.  We do not need to know the decay rate of these coefficients a priori.  The cost of our algorithm also serves as a goal for an algorithm that uses function values, which are more commonly available than Fourier series coefficients.  An important next step is to identify an essentially optimal algorithm based on function values.  Another research direction is to extend this setting to Banach spaces of inputs and/or outputs.

\bibliographystyle{elsarticle-num}
\bibliography{FJH23,FJHown23}

\end{document}